\documentclass[a4paper,12pt]{amsart}
\usepackage{amsmath,amsthm,amssymb}
\usepackage{mathptmx}
\usepackage{graphicx,graphics}
\usepackage{tikz}
\newtheorem{theorem}{Theorem}[section]

\newtheorem{prop}{Proposition}[section]
\theoremstyle{definition}
\newtheorem{definition}{Definition}[section]

\newtheorem{example}{Example}[section]
\newtheorem{problem}{Problem}[section]

\usepackage{mathrsfs}
\usepackage{enumerate}
\usepackage{float}
\usepackage{graphicx}
\usetikzlibrary{shapes.geometric,arrows}
\usetikzlibrary{mindmap}
\usetikzlibrary{patterns,shapes.misc,positioning}
\numberwithin{equation}{section}
\begin{document}

\author {A V Prajeesh}
\address{Department of Computer Science and Engineering,  Indian Institute of Information Technology Kottayam, Valavoor, Pala, 686635, India.}
 \email{prajeesh@iiitkottayam.ac.in} 

\author{ Krishnan Paramasivam}
\address{Department of Mathematics, National Institute of Technology Calicut, Kozhikode-673601, India.}
\email{sivam@nitc.ac.in}

\title{Enumertaing the distance magic labelings of a distance magic graph}
% \\ Corresponding Author : prajeesh@iiitkottayam.ac.in}


\begin{abstract}
Let $G = (V,E)$ be a graph of order $n$. A bijection $f : V \rightarrow \{1,2,\cdots,n\}$ is a distance magic labeling of $G$ if there exists a positive integer $k$ such that $\sum_{u \in N(v)}f(u) = k$ for all $v \in V$, where $N(v)$ is the neighborhood of $v$. Any graph which admits a distance magic labeling is called a distance magic graph.  In this article, we give a partial solution to the problem by Rao et al.[10] to predict all distance magic labelings of cartesian product of two cycles, $C_m \Box C_m$, where $m\equiv 2 \mod 4$. Further, we prove that the number of distance magic labelings of a distance magic graph is a multiple $| Aut(G)|$ where $Aut(G)$ is the automorphism group of the distance magic graph $G.$
\end{abstract}

\thanks{}

\maketitle
\pagestyle{myheadings}
\markboth{A.V. Prajeesh and and Krishnan Paramasivam }{Enumerataing the distance magic labelings of a distance magic graph}

\section{Introduction}

\noindent A simple \textit{graph} $G$ is an ordered pair $(V (G),E(G))$, where $V (G)$ is a nonempty finite set and $E(G)$ is a subset of unordered pairs of elements of $V(G)\times V(G)$. The elements of $V(G)$ and $E(G)$ are called the vertices and the edges of $G$, respectively. This manuscript adopts the graph-theoretic notations and terminologies from \cite{MR2368647,MR2817074}.
\par The \textit{union} of two graphs, $H_1$ and $H_2$, is the graph $H_1\cup H_2$ with the vertex set $V(H_1)\cup V(H_2)$ and edge set $E(H_1)\cup E(H_2)$. If $H_1$ and $H_2$ are disjoint, their union is referred to as the \textit{disjoint union}.
The $k^{\textnormal{th}}$ \textit{power} of a graph $G$ is a graph $G^{k}$ with the same set of vertices as $G$, and any two vertices $u$ and $v$ are connected if and only if $d_G(u,v)\leq k$.
The \textit{join} (or \textit{sum}) of two graphs, $H_1$ and $H_2$, is the graph $H_1+H_2$ with the vertex set $V(H_1)\cup V(H_2)$ and the edge set $E(H_1)\cup E(H_2)\cup \{uv: u\in V(H_1), v\in V(H_2)\}$.
% \begin{definition}
% 	The \textit{lexicographic product} of graphs $H_1$ and $H_2$ is the graph $H_1\circ H_2$, whose vertex set is $V(H_1)\times V(H_2)$, and $(h_1,h_2)(h_1',h_2')$ is an edge of $H_1\circ H_2$ if and only if $h_1h_1'\in E(H_1)$, or $h_1=h_1'$, and $h_2h_2'\in E(H_2).$
% \end{definition}
	% The \textit{direct product} of graphs $H_1$ and $H_2$ is the graph $H_1\times H_2$, whose vertex set is $V(H_1)\times V(H_2)$, and $(h_1,h_2)(h_1',h_2')$ is an edge of $H_1\times H_2$ if and only if $h_1h_1'\in E(H_1)$ and $h_2h_2'\in E(H_2).$ 
	The \textit{Cartesian product} of graphs $H_1$ and $H_2$ is the graph $H_1\Box H_2$ with the vertex set is $V(H_1)\times V(H_2)$, and $(h_1,h_2)(h_1',h_2')$ is an edge of $H_1\Box H_2$ if and only if $h_1=h_1'$ and $h_2h_2'\in E(H_2),$ or $h_2=h_2'$ and $h_1h_1' \in E(H_1).$
  Two graphs $G$ and $H$, are isomorphic, written $G\cong H$, if there exists a 
bijection $l : V (G) \rightarrow V (H)$ such that $uv \in E(G)$ if and only if $l(u)l(v) \in E(H)$; such a mapping is  an \textit{isomorphism} between $G$ and $H$. An isomorphism from $G$ to itself is an automorphism. Further, if $Aut(G)$ is the set of all automorphisms of $G$, which forms a group under composition of functions and is called the automorphism group of graph $G$. For further details one can refer \cite{MR2817074,MR1271140} and \cite{MR2583713}. Now, any automorphism of a graph $G$ is a permutation $\rho$ on the set of vertices of $G$ that preserves adjacency.  Let $G$ be graph with vertices $v_1,v_2,\cdots, v_n$. Now for every automorphism $\rho \in Aut(G)$ of the graph $G$, a unique permutation matrix $P_{\rho}$ can be associated with it and a permutation matrix is an $n\times n$ matrix that has exactly one entry of 1 in each row, and each column and all other entries are 0's. Each permutation matrix $P_{\rho}$, represents a permutation of $n$ elements. When $P_{\rho}$ matrix is multiplied with another matrix, say $B$, it results in permuting the rows (when pre-multiplying, to form $P_{\rho}B$) or columns (when post-multiplying, to form $BP_{\rho}$) of the matrix $B$. Consider a permutation $\rho : \{v_1,v_2,\cdots,v_n\} \rightarrow \{v_1,v_2,\cdots,v_n\},$ defined by $$ \rho = \begin{pmatrix}
v_1&v_2&v_3&\cdots&v_n\\
\rho(v_1)&\rho(v_2)&\rho(v_3)&\cdots&\rho(v_n) 
\end{pmatrix}.$$
Now the $n\times n$ permutation matrix $P_{\rho}=(p_{ij})$ is derived by permuting the columns of an $n\times n$ identity matrix. Thus, $p_{ij}=1$ if $\rho(v_i)=v_j$ and $p_{ij}=0$ otherwise. Now, consider a 4-cycle $C_4$. Then  \begin{figure}[htbp]
	\centering
	\includegraphics[width=30mm]{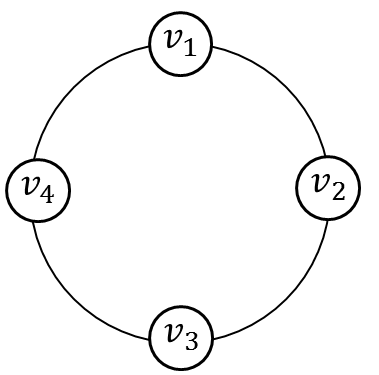}
	\caption{A $4$-cycle}
	\label{Fig2.2}
\end{figure}
$Aut(C_4)=D_8$, the dihedral group, which can represented be by the permutation group, $\mathcal{P}_1= \{\rho_0=\emptyset,\rho_1= (v_2~ v_4),\rho_2= (v_1~ v_2)(v_3~ v_4),\rho_3= (v_1~ v_2~ v_3~ v_4), \rho_4= (v_1~ v_3),\rho_5= (v_1~v_3)(v_2~ v_4),\rho_6=(v_1 ~v_4~ v_3~ v_2),$ $\rho_7= (v_1~ v_4)(v_2~ v_3)\}$.
\par Now, if we consider the automorphism $\rho_2=\begin{pmatrix}
v_1 &v_2&v_3 &v_4\\
v_2&v_1 & v_4&v_{3}
\end{pmatrix}$, then the corresponding permutation matrix $P_{\rho_{2}}$ can be obtained as,
$\begin{bmatrix}
0 & 1 & 0 & 0\\
1 & 0 & 0 & 0\\
0 & 0 & 0 & 1\\
0 & 0 & 1 & 0
\end{bmatrix}$.
\par The following result gives an idea connecting the permutation matrix corresponding to an automorphism of a graph and its adjacency matrix.
\begin{prop}\label{commute}
	\textnormal{\cite{MR1271140}} Let $A$ be an adjacency matrix of a graph $G$, and $\sigma$ a permutation of $V(G)$. Then $\sigma$ is an automorphism of $G$ if and only if $PA=AP$, where $P$ is the permutation representation of $\sigma.$
\end{prop}
\par A graph is \textit{prime} \cite{MR2817074} with respect to a given graph product if it is nontrivial and cannot be represented as the product of two nontrivial graphs.
That is, respect to the cartesian product, a nontrivial graph $H$ is said to be prime if $H = H_1\Box H_2$ implies that $H_1$ or $H_2$ is $K_1$. For example, $C_5$ is a prime graph but $C_4 \cong K_2\Box K_2$ is not a prime graph with respect to cartesian product of graphs. The following result gives a characterization regarding the automorphism group of cartesian product of connected prime graphs.
\begin{theorem}\label{Wrea}
	\textnormal{\cite{MR2817074}} The automorphism group of the Cartesian product of connected prime
	graphs is isomorphic to the automorphism group of the disjoint union of the factors.
\end{theorem}

\par	Let $\Gamma_1,\Gamma_2,\cdots,\Gamma_n$ be a finite collection of groups. The \textit{external direct
	product} of $\Gamma_1,\Gamma_2,\cdots,\Gamma_n$, written as $\Gamma_1\times \Gamma_2\times\cdots\times \Gamma_n$, is the set of all $n$-tuples for which the $i^{th}$ component is an element of $\Gamma_i$, and the operation is component-wise.
Next, the concepts related to group action and related ideas are provided.
\begin{definition}\label{action}
	\textnormal{\cite{MR0225619}} Let $X$ be a set and $\Lambda$ be a group. An action of $\Lambda$ on $X$ is a map $*: \Lambda \times X \rightarrow X$ such that,
	\begin{enumerate}
		\item $e*x = x$ for all $x \in X$.
		\item $(g_1g_2)*x = g_1*g_2*x $ for all $x \in X$ and all $g_1,g_2 \in \Lambda.$
	\end{enumerate}
\end{definition}
\par And under these conditions, $X$ is a $\Lambda$-set. Further for each $x \in X$, the orbit of $x$ in $X$ under the action of $\Lambda$ on $X$, denoted by $\Lambda x$ is defined as, $\Lambda x = \{g*x: g \in \Lambda\}$. The set $\Lambda_x$ is defined as $\{g\in \Lambda: g*x = x\}$ and $X_g$ is defined as $\{x \in X: g*x = x\}$.
\begin{theorem}\label{Gx}
	\textnormal{\cite{MR0225619}} Let $X$ be a $\Lambda$-set and let $x \in X.$ Then $|\Lambda x|=(\Lambda:\Lambda_x)$. If $|\Lambda|$ is finite, then $|\Lambda x|$ is a divisor of $|\Lambda|.$
\end{theorem}
\begin{theorem}\label{countgrp}
	\textnormal{\cite{MR0225619}} Let $\Lambda$ be a finite group and $X$ a finite $\Lambda$-set. If $r$ is the number of orbits in $X$ under $\Lambda$, then
	\begin{equation*}
		r|\Lambda| = \sum_{g\in \Lambda}|X_g|.
	\end{equation*}
\end{theorem}
\par If $\Lambda$ is a group, then the set of all automorphisms of $\Lambda$ with the binary operation as composition, will form a new group called the automorphism group of $\Lambda$. The same is denoted by $Aut\Lambda$.
\begin{definition}
	\textnormal{\cite{MR2583713}}  Given groups $\Lambda_1$ and $\Lambda_2$, and a homomorphism $\gamma : \Lambda_1 \rightarrow Aut\Lambda_2$, the (external) \textit{semi-direct product} $\Lambda_1 \rtimes_{\gamma} \Lambda_2$ of $\Lambda_2$ by $\Lambda_1$ relative to $\gamma$ is the group with underlying set $\Lambda_1 \times \Lambda_2$, and operation $(a_1, k_1)(a_2, k_2) = (a_1a_2, k_1(a_2\gamma )k_2)$, where $a_i \in \Lambda_1$ and $k_i\in \Lambda_2$, for $i = 1, 2.$
\end{definition}  
Suppose $\Lambda_1$ and $\Lambda_2$ are groups, $X = \{1, 2,\cdots, n\}$,
$\Lambda_2$ acts on $X$, and $\Lambda_1^*$ is the direct product of $n$ copies of $\Lambda_1$. We define an action of $\Lambda_2$ on $\Lambda_1^*$ by
$(g_1,\cdots, g_n)*h$ = $(g_{1{*h^{-1}}},\cdots, g_{n{*h^{-1}}})$  for $h \in \Lambda_2.$ It can be seen that the axioms in Definition \ref{action} are satisfied and,
$(g_1,\cdots, g_n)*h(g_1',\cdots , g_n')*h = (g_1g_1,' \cdots , g_ng_n')*h,$ gives a homomorphism.
\par Also, the \textit{wreath product} \cite{MR2583713} of $\Lambda_1$ by $\Lambda_2$ is defined as the semi-direct product $\Lambda_2 \rtimes_{\phi} \Lambda_1^*$
where $\phi$ is the homomorphism given above. The product
is denoted by either $\Lambda_1$ wr $\Lambda_2$ or $\Lambda_1 \wr \Lambda_2$.

\section{Distance magic labeling of graphs}
\noindent In 1994, the concept of  of distance magicness of  graphs, was introduced by Vilfred \cite{vilfred}. He called the same as $\sum$-labeling. Further, Miller et al. \cite{MR1999203} studied the same concept independently and named it as 1-vertex magic vertex labeling. Later, the same labeling was called distance magic labeling  by Sugeng et al. \cite{MR3547016}.

\begin{definition}\label{weight}
	Let $G$ be a graph on $n$ vertices. Consider $l : V(G)\rightarrow \{1,2,\cdots,n\}$ to be a bijection and $x$, any vertex of $G$. Then the \textit{weight} of $x$, denoted by $w_G(x)$, is defined to be the sum of labels of $v$, where the summation is taken over all vertices $v$ of $N_G(x)$, that is, 
	$$ w_{G}(x) = \sum_{v\in N_{G}(x)} l(v).$$
\end{definition}
\begin{definition}
	\cite{MR1999203} Let $G$ be a graph on $n$ vertices. A bijection  $l : V(G)\rightarrow \{1,2,\cdots,n\}$ is a \textit{distance magic labeling} of a graph $G$ if there exists a positive integer $\mu$ such that for any vertex $x$ of $G$,  $w_G(x)=\mu$. A graph $G$ that admits such labeling is called a distance magic graph (or shortly, distance magic graph).
\end{definition} 
A distance magic graph with magic constant 27 is given in Figure \ref{Fig3}.
	\begin{figure}[htbp]
		\centering
		\includegraphics[width=55mm]{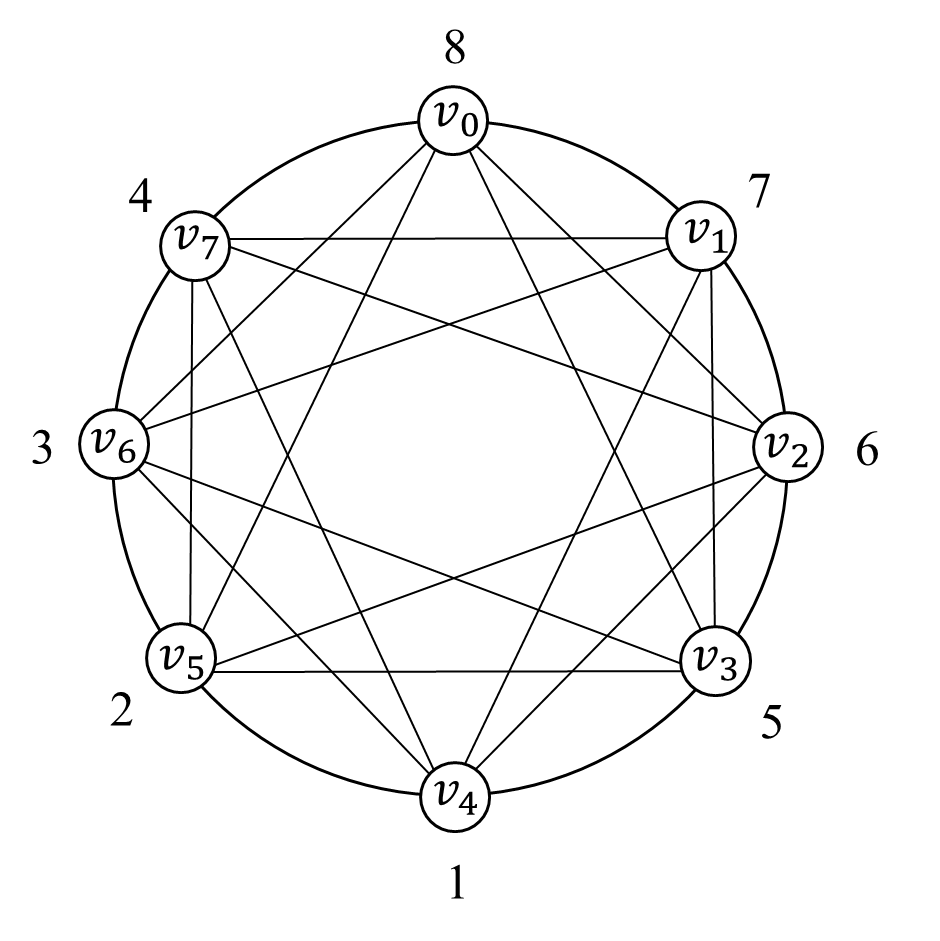}
		\caption{Distance magic graph with $\mu=27$}
		\label{Fig3}
	\end{figure}
 
In 2015, Yasin and Simanjuntak  \cite{yasin} obtained all distance magic graphs on $n$ vertices with $n \leq 9$. There are merely 16 distance magic graphs (up to isomorphism) of order $\le 9$. Which makes it clear that verifying the distance magicness of a general graph is a challenging problem to attempt.  
\par It can be seen that many open problems and conjectures are posted in this domain, but most of them remain still unsolved. In 2004, Rao et al. \cite{Rao2004}
characterized the distance magicness of the Cartesian product of two cycles, $C_n$ and $C_k$, denoted by $C_n \Box C_k$, and posted the following three problems.
\begin{problem}\label{4-reg}
	\textnormal{\cite{Rao2004}} Characterize 4-regular distance magic graphs.
\end{problem}
\begin{problem}\label{prb3}
	\textnormal{\cite{Rao2004}} Characterize graphs $G$ and $H$ for which $G\Box H$ are distance magic.
\end{problem}
\begin{problem}\label{prb}
	\textnormal{\cite{Rao2004}} Find all distance magic labelings of graph $C_m \Box C_n$, where $m = n$ and $m \equiv 2 \bmod 4.$
\end{problem}
However, while the condition given on existence of distance magic labeling for the class $C_m \Box C_n$ by Rao et al. \cite{Rao2004} was sufficient, but not necessary. The corrected result is the following which was proved by Rozman and Sparl in \cite{sparl}.

\begin{theorem}\cite{sparl}
Let $m,n$ be integers with $n \geq m \geq 3$. Then the Cartesian product of cycles
$C_m \Box C_n$ is distance magic if and only if $n = 2m$ with $m$ odd, or $n = m $ with $ m \equiv 2 \mod 4$.
\end{theorem}
The authors also addressed Problem \ref{prb} in \cite{sparl} by characterizing all distance magic labelings of such graphs in terms of the existence of appropriate sequences of integers used as labelings. Moreover, they also established a
the following result as the lower bound for the number of all distance magic labelings of $C_m \Box C_{2m}$ with $m \geq 3$.
\begin{prop}\cite{sparl}
    Let $m \geq 3$ be an odd integer and let $\Gamma=C_m \square C_{2 m}$. Then there are at least $4 m^2 \cdot\left(2^{m+1}-3\right) \cdot(m-1)!^2$ different distance magic labelings of $\Gamma$.
\end{prop}
However, we give a different approach to solve the same problem in the next section.
%------------------------------------------------
\section{Counting the number of distance magic labelings of a distance magic graph}
 A generalization of the Problem \ref{prb} is to find all the distance magic labelings of a distance magic graph. This section attempts to provide a partial solution to these problems. For more details, see \cite{prajeesh2021}. 
 Since a labeling of a graph $G$ with $n$ vertices, is a bijection from $V(G)$ to $\{1,2,\cdots,n\}$, there are a total of $n!$ number of labelings (or bijections) possible for $G$. Let $\mathcal{F}$ be the set of all labelings of a graph $G$. Also, $\mathcal{F}^d \subseteq \mathcal{F}$ be the set of all distance magic labelings of $G$. Let $A$ be the $n\times n$ adjacency matrix of $G$. Now, if $G$ is a distance magic graph, then there exists an $l\in \mathcal{F}$ with a magic constant $\mu$.
The same $l$ can be represented as a column matrix ${X}=[l(v_1) l(v_2) \cdots l(v_n)]^T$ where $l(v_i)\in \{1,2,\cdots,n\}$ with $l(v_i)\neq l(v_{j})$ whenever $i \neq j$.
Also, we can see that, if $X$ represents a distance magic labeling of $G$, then 
\begin{equation*}
AX = \mu \bar{1},
\end{equation*}
where $\mu$ is the magic constant of $G$ and $\bar{1}=[1~ 1~\cdots ~1]^T$ is $n\times 1$ column matrix.
\par \noindent Thus, if
\begin{eqnarray*}
	AX &=& \mu\bar{1}
\end{eqnarray*}
then 
\begin{eqnarray*}
	P_{\rho}AX &=& P_{\rho}\mu\bar{1}=\mu \bar{1}.
\end{eqnarray*}
Here, $P_\rho$ is a permutation matrix corresponding to an element $\rho \in Aut(G).$ From Proposition \ref{commute}, these $P_{\rho}$ matrices commute with $A$. That is, $P_{\rho}A=AP_{\rho}$. Hence,
\begin{eqnarray*}
	AP_{\rho}X &=& \mu\bar{1}
\end{eqnarray*}
or
\begin{eqnarray*}
	AY &=& \mu\bar{1}.
\end{eqnarray*}

\noindent Thus, $Y$ corresponds to a new labeling of the graph $G$. 
\par Note that these permutation matrices are invertible matrices.
Fix a distance-magic labeling $X$. For every $\sigma \in \operatorname{Aut}(G)$, we define
$$
Y_\sigma=P_\sigma X.
$$
Since $\sigma$ preserves adjacency, $Y_\sigma$ is again a distance-magic labeling.
Moreover, if
$$
P_{\sigma_1} X=P_{\sigma_2} X
$$
then
$$
P_{\sigma_2}^{-1} P_{\sigma_1} X=X.
$$
Because $X$ assigns distinct labels to vertices, the only automorphism that can leave $X$ unchanged is the identity. Hence
$$
\sigma_1=\sigma_2.
$$
Therefore the map
$$
\sigma \longmapsto P_\sigma X
$$
is injective, and
$$
\left|\mathcal{F}^d\right| \geq|\operatorname{Aut}(G)| .
$$
Consequently, for each automorphism in $Aut(G)$, we obtain a distance magic labeling $Y$. Hence, there are at least $|Aut(G)|$ number of different distance magic labelings for the given graph $G$, or equivalently,  $|\mathcal{F}^d|\geq |Aut(G)|$.
\par  The following theorem determines the number of distance magic labelings of a distance magic graph.
\begin{theorem}\label{thm3}
	Let $G$ be a distance magic graph and $\mathcal{F}^d$ be the set of all distance magic labelings of $G$ and $*:Aut(G) \times \mathcal{F}^d  \rightarrow \mathcal{F}^d$, be a group action defined by $\rho*l =  l\circ \rho^{-1}$, where $\lq\circ$' is  the composition of $l$ and $\rho^{-1}$. Then $|\mathcal{F}^d| = k|Aut(G)|$, for some $k \geq 1.$ Moreover, $k$ is the number of non-equivalent distance magic labelings of $G$.
\end{theorem}
\begin{proof}
	Let $l$ be a distance magic labeling of $G$ with magic constant $\mu$. Clearly, if $\rho$ is an automorphism of $G$, such that $\rho(v)=u,$ then $\rho$ maps $N_G(v)$ onto $N_G(u)$. Hence, if we define a new labeling, $l'(x) = l\circ \rho^{-1}(x)$ $\forall$ $x \in V(G)$, then it will be a distance magic labeling of $G$. \par Clearly, if we define $l_1 = l \circ \rho_1^{-1}$ and $l_2=l \circ \rho_2^{-1}$, then, $l_1=l_2$ implies, $l \circ \rho_1^{-1} = l \circ \rho_2^{-1}$,  or, $l \circ \rho_1^{-1} \circ \rho_1= l \circ \rho_2^{-1}\circ \rho_1$, or,  $l = l \circ (\rho_2^{-1}\circ \rho_1)$. That is, $\rho_2^{-1}\circ \rho_1 =  \rho_0$, or, $\rho_1 =  \rho_2$.
	\par Define $*:Aut(G) \times  \mathcal{F}^d \rightarrow  \mathcal{F}^d$, as $\rho*l =  l\circ \rho^{-1}$, where '$\circ$' represents the composition of $l$ and $\rho^{-1}$. Clearly, this group action is well defined. Let $\rho_1=\rho_2$ and $l_1=l_2$. Since $\rho_1,\rho_2\in Aut(G)$, we have that $\rho_1^{-1}=\rho_2^{-1}$.  Thus, $$\rho_1* l_1=l_1\circ \rho_1^{-1}=l_1\circ \rho_2^{-1}=l_2\circ \rho_2^{-1}=\rho_2* l_2.$$ \par Now, let $\rho_{0}$ be the identity in $Aut(G)$. Also, we have
	\begin{enumerate}
		\item[(a)] $\rho_{0}*l=l\circ \rho^{-1}_{0} = l,$ for all $l\in \mathcal{F}^d$.
		\item[(b)] $\rho\sigma*l=l\circ(\rho\sigma)^{-1} = l\circ (\rho \circ \sigma)^{-1} = (l\circ \sigma^{-1})\circ \rho^{-1} = (\sigma *l)\circ \rho^{-1} = \rho*\sigma*l$, for all $l\in \mathcal{F}^d$ and for all $\rho,\sigma \in Aut(G).$
	\end{enumerate}
	Clearly, $ \mathcal{F}^d$ is a $Aut(G)$-set. For any $l \in \mathcal{F}^d,$ define $Aut(G)_{l} =\{\rho \in Aut(G): \rho*l = l\}.$ Next, we claim that $Aut(G)$ acts faithfully on $ \mathcal{F}^d$. So it suffices to prove that $\rho_{0}$ is the only element  of $Aut(G)$ that leaves every element $l$ of $\mathcal{F}^d$ fixed.
	\par	If $\rho = \rho_{0}$, then $\rho_{0}*l = l\circ \rho^{-1}_{0} = l.$ Also, if $\rho*l = l\circ \rho^{-1} = l$,  which implies $ l^{-1}\circ (l\circ \rho^{-1}) = \rho_0$, or equivalently, $\rho = \rho_{0}.$	Hence, $Aut(G)_{l}=\{\rho_0\}.$ \par Now, we define a relation on $ \mathcal{F}^d$ as, $l_1 \sim l_2$ if and only if there exists $\rho \in Aut(G)$ such that $\rho * l_1 = l_2.$	
	Here for each $l \in \mathcal{F}^d$, we have $\rho_0*l = l$. Also, if $l_1 \sim l_2$, then there exists $\rho \in Aut(G)$ such that $l_2 = \rho*l_1$. That is,
	$l_2 = l_1 \circ \rho^{-1}$, or $l_2 \circ \rho = l_1 $. Hence $l_2 \sim l_1$. Further, if $l_1 = \rho_1*l_2$ and $l_2 = \rho_2 * l_3,$ then $l_1 =l_2 \circ \rho_1^{-1}= l_3 \circ \rho_2^{-1} \circ \rho_1^{-1} = l_3 \circ (\rho_1\rho_2)^{-1} = \rho_1\rho_2 * l_3$. Therefore, we get that $\sim$ is an equivalence relation on $\mathcal{F}^d$. The orbit of $l$ in $ \mathcal{F}^d$ under $Aut(G)$ is defined as, $Aut(G)l$ = $\{l*\rho: \rho\in Aut(G)\}$. Now two distance magic labelings $l_1$ and $l_2$ are equivalent if and only if the equivalent class generated by $l_1$ is equal to the equivalence class generated by $l_2$.
	\par By Theorem \ref{Gx}, we have $|Aut(G)l|=(Aut(G):Aut(G)_{l})=|Aut(G)|$.
	Now, we can see that $\mathcal{F}^d_{\rho_0} =  \mathcal{F}^d$ and for all $\rho \neq \rho_{0}$, $\mathcal{F}^d_{\rho} = \phi.$ Hence using Theorem \ref{countgrp}, we have 
	\begin{equation*}
	k|Aut(G)| = \sum_{\rho\in Aut(G)}|\mathcal{F}^d_{\rho}| = |\mathcal{F}^d|,
	\end{equation*}
	where $k$ is the number of orbits in $ \mathcal{F}^d$ under $Aut(G)$. Now, since $G$ is distance magic, there exists at least $|Aut(G)|$ number of distance magic labelings for a distance magic graph.  Now, for each distance magic labeling $l$, the orbit of $l$ contains $|Aut(G)|$ number of labelings. Hence, if $k$ is the number of non-equivalent distance magic labelings of a graph $G$, then we have $k$ orbits each having $|Aut(G)|$ number of distance magic labelings in it. Hence, we have $|\mathcal{F}^d| = k.|Aut(G)|$.
\end{proof}
\begin{example}
	Consider the graph $G \cong C_4 \cup P_3$. 
	\begin{figure}[htbp]
		\centering
		\includegraphics[width=75mm]{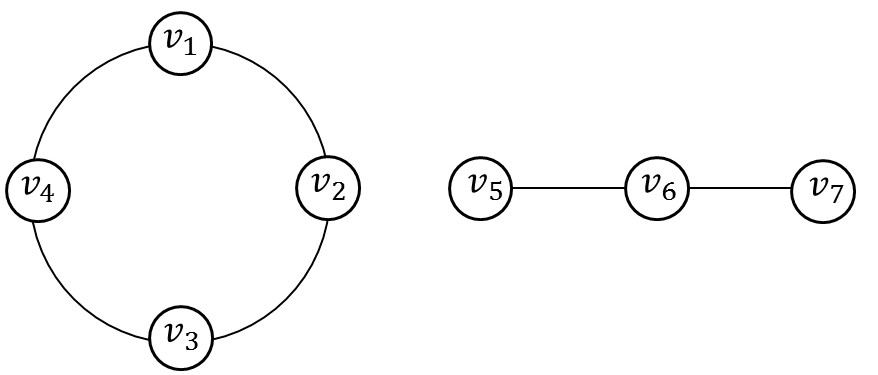}
		\label{Fig3.11}
		\caption{$C_4\cup P_3$}
	\end{figure}
	Then the automorphism group of $G$ is   \\ $\mathcal{P}_2 = \{\rho_{0}=\phi, \rho_1=(v_5~ v_7), \rho_2=(v_2~ v_4), \rho_3=(v_2~ v_4)(v_5~ v_7), \rho_4= (v_1~ v_2)(v_3~ v_4),\\ \rho_5=(v_1~ v_2)(v_3 ~v_4)(v_5~ v_7), \rho_6=(v_1~ v_2~ v_3~ v_4), \rho_7=(v_1~ v_2~ v_3~ v_4 )(v_5~ v_7), \rho_8=(v_1~ v_3), \\ \rho_9=(v_1 ~v_3)(v_5~ v_7), \rho_{10}=(v_1~ v_3)(v_2~ v_4), \rho_{11}=(v_1~ v_3)(v_2~ v_4)(v_5~ v_7), \rho_{12}=(v_1~ v_4~ v_3~ v_2),$ $\rho_{13}=(v_1~ v_4~ v_3 ~v_2)(v_5~ v_7), \rho_{14}= (v_1 ~v_4)(v_2~ v_3),$ $ \rho_{15}=(v_1~ v_4)(v_2~ v_3)(v_5~ v_7)\}$.  
	\vskip ,1cm \noindent Here, the number of non-equivalent distance magic labelings is 3, which are $l,m$ and $p.$
	\begin{enumerate}
		\item[(a)] $l(v_1)=6, l(v_2)=5, l(v_3)=1, l(v_4)=2, l(v_5)=4, l(v_6)=7, l(v_7)=3.$
		\item[(b)] $m(v_1)=5, m(v_2)=4, m(v_3)=2, m(v_4)=3, m(v_5)=6, m(v_6)=7, m(v_7)=1$.
		\item[(c)] $p(v_1)=4, p(v_2)=1, p(v_3)=3, p(v_4)=6, p(v_5)=5, p(v_6)=7, p(v_7)=2$.
	\end{enumerate}
	Under the action of $Aut(G)$ on the vertex set of $G$, the three orbits are $\{v_1,v_2,v_3,v_4\}, \\ \{v_5,v_7\}$ and $\{v_6\}$. One can see that the magic constant of $G$ is 7. Since $v_5$ and $v_7$ have a neighbour as $v_6$, the label of the vertex $v_6$ is 7, and the sum of the labels of $v_5$ and $v_7$ has to be 7. This can be obtained in three different ways $l(v_5) = 4, l(v_7)=3$; $m(v_5)=6, m(v_7)=1$ and $p(v_5)=5, p(v_7)=2.$ That is, 7 can be partitioned in to three distinct ways as $4+3$, $6+1$ and $5+2$.
	\par  Since $v_5$ and $v_7$ are in the orbit $\{v_5,v_7\}$, under the action defined in the above theorem, never $m$ can be obtained from $l$ using any automorphism. Hence they are two non-equivalent labelings. \par Similarly, we obtain only three such non-equivalent labelings. So all the other labelings are not distance magic labelings of $G.$
	\par  The total number of labelings possible are 7! = 5040. But the number of distance magic labelings is only $|\mathcal{F}^d| = 3|Aut(G)| = 48.$ However for a general graph finding the number of non-equivalent labelings seems to be challenging.
\end{example}
\subsection{Partial solution to the problem by Rao et al.\cite{Rao2004}}
\noindent From \cite{MR29492}, one can see  that the automorphism group of the graph $2G = Aut(G)\wr S_2$, the external wreath product of $Aut(G)$ with $S_2$. Now, since $C_n$ , $n \neq 4$ is a prime graph with respect to the Cartesian product of graphs, by using Theorem \ref{Wrea}, we have
\begin{equation*}
Aut(C_n \Box C_n) = D_{2n} \wr S_2,
\end{equation*}
where $D_{2n}$ is the dihedral group of order $2n$, and $S_2$ is the symmetric group on 2 elements. Note that $D_{2n} \wr S_2$ is the group of order $8n^2$. Thus one can conclude that the number of distance magic labelings of the graph $C_n \Box C_n$ is $8kn^2$, for some $k\geq 1$ and $n\equiv 2 \bmod 4$.
\par  \vskip .1cm \noindent Similarly, when $m$ is odd, for the graph $C_m \Box C_{2m}$, 
\begin{equation*}
Aut(C_m \Box C_{2m}) = Aut(C_m) \times Aut(C_{2m}).
\end{equation*}
Therefore, the total number of distance magic labelings will be $8km^2$, for some $k \geq 1$.

\par Table \ref{tabcount1} lists the number of distance magic labelings of all distance magic graphs with $p \leq 9$ with $k=1$. Table \ref{tabcount2} lists the number of distance magic labelings of all distance magic graphs with $p \leq 9$ with $k=2$. Table \ref{tabcount3} lists the number of distance magic labelings of all distance magic graphs with $p \leq 9$ with $k\geq3$. The complete list of such graphs can be obtained from \cite{yasin}. A brute force algorithm is used to verify all these numbers.
\begin{table}[htbp]
	\centering
	\caption{ Number of distance magic labelings of graphs up to 9 vertices, \& $k=1$}
	\label{tabcount1}
	\begin{tabular}{llll}
		\hline\noalign{\smallskip}
		Graph $G$ ~~~~~~~~& Regular &$|Aut(G)|$~~~~~~~~ & $|\mathcal{F}^d|$~~~~~~~~  \\
		\noalign{\smallskip}\hline\noalign{\smallskip}
		$P_3$& No& 2&2\\
		$C_4$ & Yes& 8 & 8\\
		$C_6^2$ & Yes&48 & 48\\
		$C_6^2+K_1$ & No&48 & 48 \\
		$C_4+\bar{K}_4$ & Yes& 192 & 192 \\
		$ C_8^3$ & Yes&384  &  384  \\
		$ C_8^3+K_1$ & No&384  &  384  \\
		$(K_2\cup 2K_1)+\bar{K}_5$ & No&960  & 960 \\
		$K_{2,3}+\bar{K}_4$& No&288& 288\\
		\noalign{\smallskip}\hline
	\end{tabular}
\end{table} 
\begin{table}[htbp]
	\centering
	\caption{ Number of distance magic labelings of graphs up to 9 vertices with $k=2$}
	\label{tabcount2}
	\begin{tabular}{llll}
		\hline\noalign{\smallskip}
		Graph $G$~~~~~~~~& Regular &$|Aut(G)|$~~~~~~~~ & $|\mathcal{F}^d|$~~~~~~~~  \\
		\noalign{\smallskip}\hline\noalign{\smallskip}
		$K_{3,3,3}$ & Yes&1296 & 2592\\	
		\noalign{\smallskip}\hline
	\end{tabular}
\end{table} 
\begin{table}[htbp]
	\centering
	\caption{ Number of distance magic labelings of graphs up to 9 vertices with $k\geq3$}
	\label{tabcount3}
	\begin{tabular}{llll}
		\hline\noalign{\smallskip}
		Graph $G$~~~~~~~~& Regular &$|Aut(G)|$~~~~~~~~ & $|\mathcal{F}^d|$~~~~~~~~  \\
		\noalign{\smallskip}\hline\noalign{\smallskip}
		$C_4 \cup P_3$ & No& 16  & 48  \\
		$K_{3,4}$ & No&144  & 576  \\
		$2C_4$ & Yes& 128  &  384  \\
		$K_{3,5} $& No &720  &  2160  \\
		$K_{4,4}$ & Yes& 1152  & 4608 \\
		$G^*$ & No&32 & 384			\\			
		\noalign{\smallskip}\hline
	\end{tabular}
\end{table} 
\begin{figure}[htbp]
	\centering
	\includegraphics[width=40mm]{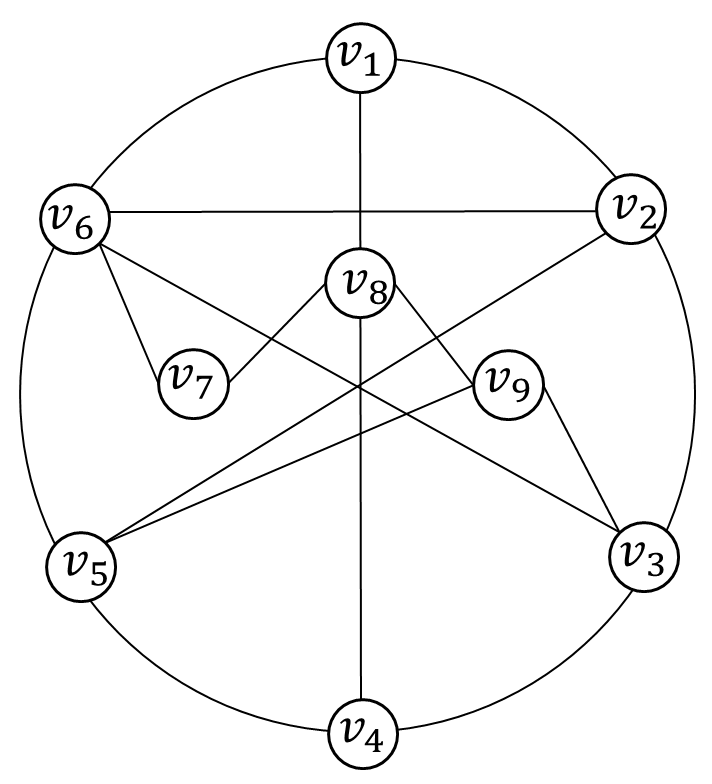}
	\caption{Graph $G^*$}
	\label{Graph $G^*$}
\end{figure}
\section{Conclusion and Future Scope}
\noindent The manuscript gives a partial solution to the problem posted by Rao et al.\cite{Rao2004} on finding all the distance magic labelings of $C_n \Box C_n$, where $n \equiv 2 \mod 4.$ Also the count for the number of distance magic labeling for the graph $C_m \Box C_{2m}$, where $m$ is odd and $m \geq3$ is also provided.  Also we predict an appropriate count of the distance magic labelings for a general distance magic graph to be a multiple of the order of a graph's automorphism group. However the problem to find the number of non equivalent (non-similar, some suitable word) labelings is still open for an arbitrary distance magic graph.

% \bibliographystyle{IEEEtran}
% %\bibliographystyle{plainnat} % use this to have URLs listed in References
% \bibliography{references}
\end{document}